\documentclass{article}
\usepackage{authblk}
\usepackage[utf8]{inputenc}
\usepackage[margin=1.1in]{geometry}
\usepackage{amsmath,amssymb,amsthm}
\usepackage{algorithm,algorithmic}
\usepackage{graphicx}
\usepackage{subcaption}
\newtheorem{proposition}{Proposition}
\input{macros.sty}
\newcommand{\ten}[1]{\boldsymbol{\mathcal{#1}}}

\usepackage{diagbox}

\begin{document}

\title{
Efficient hyperparameter estimation in Bayesian inverse problems using sample average approximation}

\author[1]{
Julianne Chung}
\author[2]{Scot M. Miller}
\author[3]{Malena Sabat\'e  Landman}
\author[4]{Arvind K. Saibaba}

\affil[1]{Department of Mathematics, Emory University, Atlanta, GA USA}
\affil[2]{Department of Environmental Health and Engineering, Johns Hopkins University, Baltimore, MD, USA}
\affil[3]{Mathematical Institute, University of Oxford, UK}
\affil[4]{Department of Mathematics, North Carolina State University, Raleigh, NC, USA}

\maketitle

\begin{abstract}
In Bayesian inverse problems, it is common to consider several hyperparameters that define the prior and the noise model that must be estimated from the data. In particular, we are interested in linear inverse problems with additive Gaussian noise and Gaussian priors defined using Mat\'{e}rn covariance models. In this case, we estimate the hyperparameters using the maximum a posteriori (MAP) estimate of the marginalized posterior distribution. However, this is a computationally intensive task since it involves computing log determinants. To address this challenge, we consider a stochastic average approximation (SAA) of the objective function and use the preconditioned Lanczos method to compute efficient approximations of the function and gradient evaluations. We propose a new preconditioner that can be updated cheaply for new values of the hyperparameters and an approach to compute approximations of the gradient evaluations, by reutilizing information from the function evaluations. We demonstrate the performance of our approach on static and dynamic seismic tomography problems.
\end{abstract}


\section{Introduction}

Inverse problems arise in many important applications, where the aim is to estimate some unknown inverse parameters from given observations or data. In particular, we consider the unknown parameters to be in the form of detailed spatial or spatial temporal maps of physical quantities, e.g., slowness of the medium in seismic tomography. The inverse problems that we tackle in this paper are ill-posed; which means that a solution to the problem may not exist, may not be unique, or may not depend continuously on the data. To address the ill-posedness, we adopt the Bayesian approach to find  approximate solutions to the inverse problems.

In the Bayesian approach, we treat the unknown parameters as a random variable and impose a prior distribution on it. The forward operator, which maps the parameters to the observations, is assumed to be fixed; and the data is assumed to be corrupted by additive Gaussian noise. These ingredients define the likelihood. The Bayesian approach uses the Bayes rule to combine the prior and the likelihood to give the posterior distribution, which characterizes the uncertainty in the unknown parameters given measurement data. In this approach, there are several hyperparameters to be considered that govern the prior and the noise model, that are typically unknown in practice and must be estimated from the data. We follow the hierarchical Bayes approach and treat the hyperparameters as random variables as well, so we impose hyperpriors on them. Thus, now the posterior distribution depends jointly on the unknown parameters and the hyperparameters conditioned on the observed data.

The joint posterior distribution is non-Gaussian in general, and a full exploration of this distribution is computationally challenging for several reasons: the number of unknown parameters is large and evaluating the likelihood can be very expensive in practice. A typical approach to explore the posterior distribution, such as Markov Chain Monte Carlo (MCMC), can be infeasible in such situations.

In this paper, we consider the case where the forward problem is linear and the posterior distribution on the unknown parameters is Gaussian. The setting is similar to our recent work~\cite{hall2023efficient}, where we considered the marginal posterior distribution, obtained from the joint posterior distribution by marginalizing the unknown parameters, and compute the mode or maximum a posteriori (MAP) estimate of the hyperparameters. Since this is still computationally expensive for large-scale problems, the approach in~\cite{hall2023efficient} was to use a low-rank approximation of the forward operator using iterative methods. However, the low-rank approximation is not always suitable. Therefore, we propose a different approach based on stochastic optimization.    

\paragraph{Contributions} We propose a sample average approximation (SAA) type method for computing the maximum a posteriori (MAP) estimate of the marginalized posterior distribution. The novel contributions of this paper are as follows.
\begin{enumerate}
    \item The method to estimate the objective function combines a Monte Carlo estimator for the log-determinant of the matrix with a preconditioned Lanczos approach to apply the matrix logarithm.  We analyze the impact of the number of Monte Carlo samples and Lanczos iterations on the accuracy of the log-determinant estimator. 
    \item We use a novel preconditioner to accelerate the Lanczos iterations. The preconditioner is based on a parametric low-rank approximation of the prior covariance matrix, that is easy to update for new values of the hyperparameters. In particular, no access to the forward/adjoint solver is needed to update the preconditioner, and only a modest amount of precomputation is needed as a setup cost (independent of the optimization). 
    \item We also use a new trace estimator to approximate the gradient that has two features: first, it works with a symmetric form of the argument inside the trace, and second, it is able to reuse Lanczos iterates from the objective function computations. Therefore, the gradient can be computed essentially for free (i.e., requiring no additional forward/adjoint applications). 
    
\end{enumerate}
We demonstrate the performance of our approach on model problems from static and dynamic seismic tomography. 
\paragraph{Related work} 
The problem of optimizing for hyperparameters is closely related to parameter estimation in Gaussian processes via maximum likelihood (We may think of it as setting the forward operator as the identity matrix.). The literature on this topic is vast, but we mention a few key references that are relevant to our approach. In~\cite{anitescu2012matrix}, the authors propose a 
matrix-free approach to estimate the hyperparameters. The authors also use a sample average approximation for either the optimization problem or its optimality conditions. In~\cite{anitescu2017inversion}, the authors propose a reformulation of the problem that avoids computing the inversion of the (prior) covariance matrix. Approaches based on hierarchical matrices are considered in~\cite{geoga2020scalable,minden2017fast,ambikasaran2015fast}. Preconditioned Lanczos methods for estimating the log-determinant and its gradient are considered in~\cite{dong2017scalable,gardner2018gpytorch}. The methods we develop in this paper have some similarity to existing literature and share certain techniques in common. However, the main difference is that the Gaussian process methods do not involve forward operators. This raises two issues: first, we have to account for the problem structure which is different from Gaussian processes, and second, we have to account for the computational cost of the forward operator (and its adjoint), which may be comparable or greater than the cost of the covariance matrices. 

On the inverse problem side, there have been relatively few works for computing the hyperparameters by optimization. Several works (e.g.,~\cite{bardsley2018computational}) instead use sampling methods (e.g., MCMC), but these methods are extremely expensive since they require several thousand evaluations of the likelihood to achieve accurate uncertainty estimates.  Lastly, in~\cite{hall2023efficient}, we develop efficient methods for hyperparameter estimation based on low-rank approximations using the generalized Golub-Kahan iterative method. A brief review of other techniques is also given in the same paper. 

\paragraph{Outline} In Section \ref{sec:background} we provide an overview of the hierarchical Bayesian inverse problem and describe the optimization problem arising from the hyperparameter estimation.  Computational methods for hyperparameter estimation are described in Section \ref{sec:saa}, and numerical results are provided in Section \ref{sec:numerics}.  Conclusions and discussion are provided in Section \ref{sec:conclusions}.

\section{Hierarchical Bayesian Inverse Problems}
\label{sec:background}
In this section, we review background material on hierarchical Bayesian inverse problems. The inverse problem involves recovering the parameters $\bfs \in \bbR^n$ from measurements $\bfd$, which have been corrupted by additive Gaussian measurement noise, $\bfeta \in \bbR^{m}$, and takes the form
\begin{equation}
    \bfd = \bfA \bfs + \bfeta,  \qquad \bfeta \sim \calN((\bfzero, \bfR(\bftheta))
\end{equation}
where $\bfA \in \bbR^{m \times n}$ represents the forward map and $\bftheta  \in \bbR_+^{K}$, represents the (nonnegative) hyperparameters. In the hierarchical Bayes approach, we treat $\bftheta$ as a random variable, which we endow with prior density $\pi_{\rm hyp}(\bftheta)$. We assume that the noise covariance matrix  $\bfR : \bbR_+^{K} \, \rightarrow \bbR^{m \times m},$ where $\bfR(\cdot)$ is symmetric and positive definite (SPD), and has an inverse and square root that is computationally easy to obtain for any input (e.g., a diagonal matrix). 

\paragraph{Prior distribution} We assume that the prior distribution for the parameters $\bfs$ is also Gaussian of the form $\calN (( \bfmu(\bftheta), \bfQ(\bftheta)),$ where $\bfmu: \bbR_+^{K} \, \rightarrow \bbR^{n}$ and $\bfQ : \bbR_+^{K} \, \rightarrow \bbR^{n \times n},$ where $\bfQ(\cdot)$ is assumed to be SPD. Given a set of points $\{\bfx_j\}_{j=1}^n$ at which the unknowns $\bfs$ are represented, the covariance matrix $\bfQ(\bftheta)$ is defined through means of a covariance kernel $\kappa(\cdot, \cdot; \bftheta): \bbR^d \times\bbR^d \rightarrow  \bbR_+$. More precisely, the entries of the matrix $\bfQ$ take the form $[\bfQ (\bftheta)]_{ij} = \kappa(\bfx_i,\bfx_j;\bftheta)$ for $1\le i,j \le n. $ The class of covariance kernels we focus on in this paper is the Mat\'ern covariance class, for which we give details in Section~\ref{sec:numerics}. The resulting covariance matrices are dense, so storing and computing with them is challenging. We assume throughout this paper that matrix vector products (matvecs) with $\bfQ(\bftheta)$ can be performed efficiently; in particular, matvecs can be done in $\mc{O}(n)$ or $\mc{O}(n\log n)$ floating point operations (flops). Details are given in~\cite{ambikasaran2013fast}. For the hyperpriors, following~\cite{bardsley2018computational}, we take the prior distribution on $\bftheta$ to be a  Gamma prior with density
\begin{equation}\label{eq:hyperprior}
\pi_{\rm hyp}(\bftheta) \propto \exp\left( - \sum_{j=1}^K\gamma \theta_i \right) \qquad \theta_i > 0, \quad 1 \leq i \leq K. 
\end{equation}
The parameter $\gamma$ is set to be $10^{-4}$ and chosen such that the probability density function is relatively flat over the parameter space.

\paragraph{Posterior distribution} 

Using Bayes' theorem, the posterior density, $\pi (\bfs, \bftheta \, | \, \bfd),$ is characterized by
\[ \pi (\bfs, \bftheta \, | \, \bfd) = \frac{\pi (\bfd \, | \, \bfs, \bftheta) \pi (\bfs \, | \, \bftheta) \pi_{\rm hyp} (\bftheta)}{\pi(\bfd)}, \]
where $\pi(\cdot)$ denotes an arbitrary probability density of its argument. Using the above assumptions, the posterior density can be expressed explicitly as
\begin{equation} \pi(\bfs, \bftheta \, | \, \bfd) \propto \frac{\pi_{\rm hyp}(\bftheta) \exp \left( - \frac{1}{2} \| \bfd - \bfA \bfs \|^{2}_{\bfR^{-1}(\bftheta)} - \frac{1}{2} \| \bfs - \bfmu(\bftheta) \|^{2}_{\bfQ^{-1}(\bftheta)} \right)}{\det(\bfR(\bftheta))^{1/2} \det(\bfQ(\bftheta))^{1/2}}, 
\end{equation}
where $\| \bfx \|^{2}_{\bfK} = \bfx^{\top} \bfK \bfx$ for any SPD matrix $\bfK.$ The marginal posterior density is obtained as $\pi(\bftheta|\bfd) = \int_{\bbR^n } \pi(\bfs, \bftheta \, | \, \bfd) d\bfs$ and takes the form 
\begin{equation}
\label{eq:marginal}
\pi(\bftheta \, | \, \bfd) \propto \pi_{\rm hyp}(\bftheta) \det(\bfPsi(\bftheta))^{-1/2} \exp \left( - \frac{1}{2} \| \bfA \bfmu(\bftheta) - \bfd \|^{2}_{\bfPsi^{-1}(\bftheta)} \right), 
\end{equation}
where $\bfPsi : \bbR_+^K \rightarrow \R^{m \times m}$ takes the form
\begin{equation}
\label{eq:Zmatrix}
\bfPsi(\bftheta) = \bfA\bfQ(\bftheta)\bfA\t + \bfR(\bftheta).
\end{equation}

\paragraph{Hyperparameter estimation} One approach is to draw samples (e.g., using MCMC)  from \eqref{eq:marginal}, and using the samples to quantify the uncertainty in the hyperparameters.  However, this may be prohibitive for large-scale problems because evaluating the density function (or its logarithm) requires evaluating the determinant of and multiple solves with the matrix $\bfPsi$ that depends on $\bftheta$, which can be expensive. To compound matters, hundreds or thousands of samples are required to get accurate statistics, which can involve several hundred thousand density function  evaluations. 

Instead, the approach we follow in this paper is based on the empirical Bayes approach. In this approach, we compute the MAP estimate, which corresponds to finding the point estimate that maximizes the marginal posterior distribution. Equivalently, this estimate minimizes the negative log of the marginal posterior. That is, the problem of hyperparameter estimation becomes solving an optimization problem:
\begin{equation}
\label{eq:fullopt}
    \min_{\bftheta \in \R^K_+} \mc{F}(\bftheta) \equiv -\log\pi_{\rm hyp}(\bftheta) + \frac12 \logdet(\bfPsi(\bftheta)) + \frac12 \|\bfA \bfmu(\bftheta) - \bfd\|_{\bfPsi(\bftheta)^{-1}}^2.
\end{equation}
Notice that for each set of parameters $\bftheta$, the objective function $\mc{F}$ can be expensive to compute.  In particular, the main computational cost corresponds to evaluating the log determinant of $\bfPsi(\bftheta)$.  

For reference, we provide an analytical expression for the gradient $\nabla \mathcal{F} = \left(\frac{\partial \mathcal{F} }{\partial \theta_i}\right)_{1 \leq i \leq K},$ 

\begin{align}
    (\nabla \mathcal{F}  (\bftheta))_{i} = &  - \frac{1}{\pi_{\rm hyp}(\bftheta)} \frac{\partial \pi_{\rm hyp}(\bftheta)}{\partial \theta_{i}} + \frac{1}{2} \trace\left( \bfPsi(\bftheta)^{-1} \frac{\partial \bfPsi(\bftheta)}{\partial \theta_i} \right) \nonumber \\
    & \hspace{.4cm} - \frac{1}{2} \left[ \bfPsi(\bftheta)^{-1}(\bfA \bfmu(\bftheta) - \bfd) \right]^{\top} \left[ \frac{\partial \bfPsi(\bftheta)}{\partial \theta_i} \bfPsi^{-1}(\bftheta)(\bfA \bfmu(\bftheta) - \bfd) - 2 \bfA \frac{\partial \bfmu(\bftheta)}{\partial \theta_i} \right]. \label{eq:gradient} 
\end{align}
In the following, we describe an approach based on the sample average approximation (SAA) method for solving the optimization problem~\eqref{eq:fullopt}.

\section{Methods for estimating hyperparameters}
\label{sec:saa}
In this section, we reinterpret the optimization problem~\eqref{eq:fullopt} as a stochastic optimization problem and derive an SAA method to compute solutions efficiently. 
\subsection{SAA approach} We first interpret the objective function \eqref{eq:fullopt} as the expected value of a random variable, and then we use Monte Carlo methods to approximate the objective function. 
To this end, since $\bfPsi (\bftheta)$ is an SPD matrix for any $\bftheta$, we can write 
\begin{equation}
\label{eq:logdet}
    \logdet(\bfPsi(\bftheta)) = \trace(\log(\bfPsi(\bftheta))).
\end{equation}
Next, for any random vector $\bfw$ that is isotropic (i.e., it satisfies  $\expect[\bfw] =\bfzero$ and $\expect[\bfw\bfw\t] =\bfI$), we can write $\trace(\bfM) = \expect_{\bfw}[\bfw\t\bfM\bfw]$. Therefore, now,  
$ \logdet(\bfPsi(\bftheta)) = \expect_\bfw[\bfw\t\log (\bfPsi(\bftheta))\bfw ],$ where $\log(\bfA)$ denotes the matrix logarithm of an SPD matrix $\bfA$. The optimization problem \eqref{eq:fullopt} can be expressed as the stochastic optimization problem
\begin{equation}
\label{eqn:stochopt}
    \min_{\bftheta \in \R^K_+} -\log\pi_{\rm hyp}(\bftheta) + \frac12 \expect_\bfw[\bfw\t\log(\bfPsi(\bftheta))\bfw ] + \frac12 \|\bfA\bfmu - \bfd\|_{\bfPsi(\bftheta)^{-1}}^2 .
\end{equation}

There are two main classes of methods~\cite{shapiro2021lectures} for solving stochastic optimization problems like \eqref{eqn:stochopt}. Stochastic approximation methods are iterative methods, where random samples (e.g., one or a batch) are used to update the solution at each iteration.  SAA methods represent another family of stochastic optimization methods, which use Monte Carlo simulations.  That is, assume we have a random sample of $n_{\rm mc}$ independent realizations of $\bfw$: $\bfw_1,\dots,\bfw_{n_{\rm mc}}$. In practice, for $\bfw$ we either use Rademacher random vectors (entries independently drawn from $\{-1,+1\}$ with equal probability) or standard Gaussian random vectors (entries independently drawn from $\mc{N}(0,1)$). Then, the Monte Carlo approach replaces the expected objective function in \eqref{eqn:stochopt} with a sample average approximation of the form
\begin{equation}\label{eq:MCobjfun}
    \widehat{\mc{F}}_{\rm mc} (\bftheta) \equiv - \log\pi_{\rm hyp} (\bftheta) + \frac{1}{2n_{\rm mc}}\sum_{t=1}^{n_{\rm mc}} \bfw_t\t\log(\bfPsi(\bftheta))\bfw_t   + \frac12 \|\bfA\bfmu(\bftheta) - \bfd\|_{\bfPsi(\bftheta)^{-1}}^2. 
\end{equation}
The statistical properties of the SAA methods have been established in~\cite[Chapter 5]{shapiro2021lectures}.

However, evaluating $\widehat{\mc{F}}_{\rm mc} (\bftheta)$ can be still very expensive since it requires computing $\log(\bfPsi(\bftheta))$, and the number of required samples $n_{\rm mc}$ can be large. Moreover, in order to use efficient optimization methods, we need to either compute the gradient $\nabla \widehat{\mc{F}}_{\rm mc}(\bftheta)$  or a suitable approximation of the original gradient $\widehat{\nabla \mc{F}}_{\rm mc}(\bftheta)$.
To mitigate the first issue, several approaches are available based on either a Chebyshev or Lanczos polynomial approximation to the matrix logarithm \cite{han2017approximating,ubaru2017fast}. However, the degree of the required polynomial may still be too large to obtain accurate approximations. In this paper, to address both issues, we use a preconditioned Lanczos approach to approximate the Monte Carlo estimator for the function \eqref{eq:MCobjfun}, and we follow an analogous Monte Carlo approach for the gradient of the original function, which will be summarized later.   

For the remainder of this section, we drop the explicit dependence on $\bftheta$ when possible to simplify the exposition.

\subsection{Preconditioned log-determinant estimator}

Consider the computation of the log determinant of an SPD matrix $\bfPsi \in \bbR^{m \times m}$. Assume we have a preconditioner $\bfG$ that is (a) easy to invert, (b) satisfies $\bfG\t \bfG \approx \bfPsi^{-1}$, and (c) its determinant can be readily computed, then we can write
 \[ \logdet(\bfPsi) = \logdet(\bfG\bfPsi\bfG\t) -  2\log |\textsf{det}(\bfG)|. \]
In Section~\ref{sec:saa}\ref{ssec:precond}, we show how to efficiently compute such a preconditioner. 
 The approach then is to apply the Monte Carlo estimator to the preconditioned matrix $\bfG\bfPsi\bfG\t$ rather than the matrix $\bfPsi$. 
Therefore, the preconditioned Monte Carlo estimator to the objective function becomes
\begin{equation}
\label{eq:MCobjfun2}
\widehat{\mc{F}}_{\rm prec} = - \log\pi_{\rm hyp}  + \frac{1}{2n_{\rm mc}}\sum_{t=1}^{n_{\rm mc}} \bfw_t\t\log(\bfG\bfPsi\bfG\t)\bfw_t - \log|\det(\bfG)|  + \frac12 \|\bfA\bfmu - \bfd\|_{\bfPsi^{-1}}^2. 
\end{equation}
The resulting estimator is unbiased since 
\[ \expect\left[\frac{1}{n_{\rm mc}}\sum_{t=1}^{n_{\rm mc}} \bfw_t\t\log(\bfG\bfPsi\bfG\t)\bfw_t\right] - 2\log|\det(\bfG)| = \logdet(\bfG\bfPsi\bfG\t) -  2\log|\det(\bfG)| =  \logdet(\bfPsi).\]

Next, we show how to efficiently compute the quadratic form  $\bfw\t\log(\bfG \bfPsi \bfG \t)\bfw$, for a given nonzero vector $\bfw$. Given $\bfPsi$ and a preconditioner $\bfG$, we first compute the Lanczos recurrence with the starting vector $\bfw/\|\bfw\|_2$.  After $k$ steps of the symmetric Lanczos process, we have the matrix $\bfV_{k}= [\bfv_1,\dots,\bfv_{k}]\in \bbR^{m \times k}$ that contains orthonormal columns and tridiagonal matrix
\begin{equation}\label{eqn:tk} \bfT_k = \bmat{\gamma_1 & \delta_2  \\ \delta_2 & \gamma_2 & \delta_2 \\  &  \ddots & \ddots & \ddots \\ & &  \delta_{k-1}& \gamma_{k-1} & \delta_k \\ & & & \delta_k& \gamma_k} \in \bbR^{k \times k}  \end{equation}
that satisfy, in exact arithmetic, the following relation,
\begin{equation}\label{Lancrec}
    \bfG\bfPsi\bfG\t \bfV_k = \bfV_k\bfT_k + \delta_{k+1}\bfv_{k+1}\bfe_k\t,
\end{equation}
where $\bfe_j$ is the $j$th column of the identity matrix of appropriate size.
The preconditioned Lanczos process is summarized in Algorithm \ref{alg:lanczos}. In practice, the vectors $\bfV_k$ tend to lose orthogonality in floating point arithmetic~\cite{chow2014preconditioned} and, therefore, we use full reorthogonalization of the vectors to maintain orthogonality.  

\begin{algorithm}[!ht]
    \begin{algorithmic}[1]
     \REQUIRE Matrix $\bfPsi$, Preconditioner $\bfG$, vector $\bfw$, number of iterations $k$
        \STATE $\delta_1 = \|\bfw\|_2, \bfv_1 = \bfw/\delta_1$
\FOR {$i=1, \dots, k$}
\STATE $\gamma_i = \bfv_i\t \bfG \bfPsi \bfG\t \bfv_i $,
\STATE $\bfr = \bfG \bfPsi \bfG\t \bfv_i- \gamma_i \bfv_i- \delta_{i-1} \bfv_{i-1}$
\STATE $\bfv_{i+1} = \bfr/\delta_i$, where $\delta_i = \| \bfr\|_2$
\ENDFOR
\RETURN Matrices $\bfV_k$ and $\bfT_k$
    \end{algorithmic}
    \caption{Preconditioned Lanczos: $[\bfV_k,\bfT_k] = $Lanczos$(\bfPsi,\bfG,\bfw,k)$}
    \label{alg:lanczos}
\end{algorithm}

Next, given the Lanczos recurrence \eqref{Lancrec}, the quadratic term involving the logarithm in \eqref{eq:MCobjfun2} can be approximated as 
\begin{equation}\label{eqn:lanczoslog} \bfw\t\log(\bfG \bfPsi \bfG\t)\bfw \approx \|\bfw\|_2^2\bfe_1\t \log(\bfT_k)\bfe_1.\end{equation}
Here, the matrix logarithm can be computed using an eigendecomposition for a cost of $\mc{O}(k^3)$ flops. Alternatively, the quadratic form $\bfe_1\t \log(\bfT_k)\bfe_1,$ can be expressed in terms of a quadrature formula. Details of this can be found in~\cite{ubaru2017fast}. An algorithm to compute the objective function $\widehat{\mc{F}}_{\rm prec}$ is provided in Algorithm~\ref{alg:MCobjgrad}. 

\paragraph{Error analysis} With an appropriate choice of  preconditioner, we expect much fewer than $m$ Lanczos iterations to be required to evaluate the function accurately. 
The following result quantifies the number of samples and the number of Lanczos basis vectors required to ensure that a desired absolute error holds with high probability. To simplify notation, we denote the preconditioned matrix $\bfPsi_\bfG = \bfG\bfPsi\bfG\t$. Furthermore, let $\omega_2(\bfPsi_\bfG) = \|\bfPsi_\bfG\|_2 \|\bfPsi_\bfG^{-1} \|_2$ denote the 2-norm condition number of $\bfPsi_G$ and let $\bfL = \log(\bfPsi_\bfG) - \diag(\log(\bfPsi_\bfG))$.
\begin{proposition}
Let $\epsilon,\delta \in (0,1)$ denote two fixed user-defined parameters representing the absolute error and failure probability respectively.  Furthermore, let the number of samples $n_{\rm mc}$ and the number of Lanczos iterations, $k$, be chosen to satisfy the inequalities 
\begin{eqnarray}
n_{\rm mc}   \geq &\>  32\epsilon^{-2}\left( \|\bfL\|_2^2 + (\epsilon/2) \|\bfL\|_2 \right)\log \frac{2}{\delta},\\
k  \geq & \> \frac{\sqrt{\omega_2(\bfPsi_\bfG)+ 1}}{4}\log \left( 4\epsilon^{-1} m (\sqrt{\omega_2(\bfPsi_\bfG)}+ 1) \log (2 \omega_2(\bfPsi_\bfG))\right).
\end{eqnarray}
Then, with probability at least $1-\delta$
\[ \left|  \frac{1}{n_{\rm mc}}\sum_{t=1}^{n_{\rm mc}} \bfw_t\t\log(\bfPsi_{\bfG})\bfw_t  - 2 \log |\det(\bfG)| - \logdet(\bfPsi) \right| \le \epsilon,   \]
where $\{\bfw_t\}_{t=1}^{n_{\rm mc}}$ are independent Rademacher random vectors. 
\end{proposition}
\begin{proof}
Follows readily from~\cite[Theorem 20]{cortinovis2022randomized}.
\end{proof}

This result gives conditions on the number of samples and the number of Lanczos basis vectors needed such that the absolute error in $\logdet(\bfPsi)$ is smaller than $\epsilon$ with high probability $1-\delta$. The bound also clearly highlights the role of the conditioning of the preconditioned operator. If the condition number,  $\omega_2(\bfPsi_\bfG)$, is small, then a smaller number of Lanczos iterations $k$ is necessary for the same accuracy with high probability. Next, if $\bfG\t\bfG \approx \bfPsi^{-1}$, then $\bfPsi_\bfG \approx \bfI$ and $\bfL \approx \bfzero$, so fewer samples are required for the same accuracy. This motivates the use of a good preconditioner for ensuring efficient and accurate approximations of the objective function.

\subsection{Choice of Preconditioner} \label{ssec:precond}
To explain the computational benefits of our preconditioner, it is helpful to briefly reintroduce the dependence of $\bfPsi$ on $\bftheta$. We assume that the prior covariance has a parametric low-rank approximation of the form
\begin{equation}
\bfQ(\bftheta) \approx \bfU \bfM(\bftheta)\bfU\t,
\end{equation}
where $\bfU \in \bbR^{n \times r}$ is a fixed set of basis vectors for representing the prior covariance matrix and $\bfM(\bftheta) \in \bbR^{r \times r}$ is a symmetric and positive semidefinite matrix. We construct such a parametric low-rank approximation using the Chebyshev polynomial approximation~\cite{fong2009black,khan2024parametric}. The details of this derivation are given in Appendix \ref{app:paramlowrank}. As an alternative to the Chebyshev approximation, one could also use the approach in~\cite{shustin2022gauss} based on Gauss-Legendre features. In both cases, we can construct an approximation to the matrix $\bfPsi(\bftheta)$ as 
\[\bfPsi(\bftheta)  \approx \widehat{\bfPsi}(\bftheta) \equiv (\bfA\bfU) \bfM (\bftheta)(\bfA\bfU)\t +\bfR(\bftheta).
\]
The important point is that $\bfA \bfU$ can be constructed in advance. More precisely, computing the preconditioner for new values of $\bftheta$ after initialization, does not require any matvecs with the forward operator (or its adjoint). 

To explain how to construct the preconditioner $\bfG$, we once again drop the dependence on $\bftheta$. We form the matrix $\bfK = 
 \bfR^{-1/2}\bfA\bfU\bfM^{1/2}$ and compute its thin SVD $\bfK = \bfW\bfSigma\bfZ\t$. Then, using the Woodbury identity, we can factorize 
\[ \widehat{\bfPsi}^{-1} = \bfR^{-1/2}(\bfW\bfSigma^2\bfW\t + \bfI )^{-1} \bfR^{-1/2} = \bfG\t\bfG,   \]
where $\bfG\t = \bfR^{-1/2} (\bfI-\bfW\bfD\bfW\t)$ with $\bfD = \bfI + (\bfI+\bfSigma^2)^{-1/2}$. 
Furthermore, we can readily compute the log determinant of the preconditioner as
\begin{align*}
    \logdet(\bfG) & = \logdet(\bfR^{-1/2}) + \logdet(\bfI - \bfD) \\
    & = - \frac12\logdet(\bfR) + \logdet(\bfI - \bfD).
\end{align*}
Therefore, assuming $\bfA\bfU$ has been precomputed, we can compute the matrix $\bfG$ (given in terms of $\bfW$ and $\bfD$) in $\mc{O}(mr^2 + r^3)$ flops, see Algorithm \ref{alg:precond}. Computing a matvec $\bfG\bfx$ or $\bfG\t\bfx$ then only requires $\mc{O}(mr)$ flops and the log-determinant $\logdet(\bfG)$ an additional $\mc{O}(m)$ flops. 

\begin{algorithm}[!ht]
\begin{algorithmic}[1]
    \REQUIRE Matrix $\bfA$, instance of hyperparameter $\bftheta$
    \STATE \COMMENT{Offline stage}
    \STATE Compute matrix $\bfU$ and the information needed to compute $\bfM$ (see Appendix~\ref{app:paramlowrank})
    \STATE Compute $\bfA\bfU$ \COMMENT{$r$ forward solves}
    \STATE \COMMENT{Online stage: given an instance $\overline{\bftheta}$}
    \STATE Compute $\bfM = \bfM(\overline{\bftheta})$ and $\bfR = \bfR(\overline\bftheta)$
    \STATE Compute $\bfK =\bfR^{-1/2} (\bfA\bfU)\bfM^{1/2} $ and its thin SVD $\bfW\bfSigma\bfZ\t$
    \STATE Compute diagonal matrix $\bfD = \bfI + (\bfI + \bfSigma^2)^{-1/2}$
    \RETURN Matrices $\bfW$ and $\bfD$
\end{algorithmic}
\caption{Compute preconditioner: $[\bfW,\bfD]=$Precond($\bfA,\bftheta$)}
\label{alg:precond}
\end{algorithm}

\subsection{Gradient computations} 
To compute the derivative of $\widehat{\mc{F}}_{\rm prec}$, we have to differentiate the matrix logarithm. One possible way of tackling this is to use the approach in~\cite{konig2023efficient}. However, we adopt a different approach in this paper that allows us to estimate the gradients $\{\widehat{\nabla \mc{F}}_i\}_{i=1}^K$ with only a little additional computational cost. Note that in this approach, we are not computing exact derivatives of $\widehat{\mc{F}}_{\rm prec}$ or of $\mc{F}$, but we are constructing Monte Carlo approximations of the gradient of $\mc{F}$. 

Consider the expressions for the derivative in~\eqref{eq:gradient}. Recall that during the computation of the objective function, we already have computed the solution $\bfz$ to $\bfPsi\bfz = \bfA\bfmu-\bfd$. Therefore, we can evaluate the expression (for $1 \le i \le K$) 
\[ \left[ \bfPsi^{-1}(\bfA \bfmu - \bfd) \right]^{\top} \left[ \frac{\partial \bfPsi}{\partial \theta_i} \bfPsi^{-1}(\bfA \bfmu - \bfd) - 2 \bfA \frac{\partial \bfmu}{\partial \theta_i} \right] = \bfz\t \frac{\partial \bfPsi}{\partial \theta_i} \bfz - 2 (\bfA\t\bfz)\t\frac{\partial \bfmu}{\partial \theta_i}, \]
efficiently by taking advantage of the right-hand side in the above displayed equation. Therefore, we focus our attention on approximating the term $\trace(\bfPsi^{-1}\frac{\partial \bfPsi}{\partial \theta_i} )$ in~\eqref{eq:gradient}. Applying a na\"ive Monte Carlo estimate $\frac{1}{n_{\rm mc}}\sum_{t=1}^{n_{\rm mc}} \bfw_t\t \bfPsi^{-1}\frac{\partial \bfPsi}{\partial \theta_i}\bfw_t$ requires us to solve  linear systems of the form $\bfPsi\bfz_t = \bfw_t$, which incurs an additional expense. Instead, we use the symmetric factorization
\[ \bfPsi^{-1} = \bfG^{-\top} (\bfG\bfPsi\bfG\t)^{-1/2}(\bfG\bfPsi\bfG\t)^{-1/2} \bfG^{-1}, \]
and the cyclic property of the trace estimator to use the Monte Carlo trace estimator 
\[ \begin{aligned} \trace(\bfPsi^{-1}\frac{\partial \bfPsi}{\partial \theta_i} ) = & \> \trace\left((\bfG\bfPsi\bfG\t)^{-1/2} \bfG^{-1}\frac{\partial \bfPsi}{\partial \theta_i}\bfG^{-\top} (\bfG\bfPsi\bfG\t)^{-1/2}\right)  \\
\approx & \>  \frac{1}{n_{\rm mc}}\sum_{t=1}^{n_{\rm mc}} \bfzeta_t\t\frac{\partial \bfPsi}{\partial \theta_i} \bfzeta_t, \qquad 1 \le i \le K, 
\end{aligned} 
\] 
where $\bfzeta_t \equiv\bfG^{-\top} (\bfG\bfPsi\bfG\t)^{-1/2} \bfw_t$, for $1 \le t \le n_{\rm mc}$. Recall that, to approximate the evaluation of the objective function, we had already computed an approximation to $\bfw_t\t\log(\bfG\bfPsi\bfG\t)\bfw_t$ in~\eqref{eqn:lanczoslog} using Lanczos to construct $\bfV_k$ and $\bfT_k$. If we consider reusing the same basis vectors, we can approximate $\bfzeta_t$ cheaply using the formula
\begin{equation}\label{eqn:lanczossqrt}
    \bfzeta_t \approx \|\bfw_t\|_2 \bfG^{-\top} (\bfV_k\bfT_k^{-1/2}\bfe_1).
\end{equation}
Using this characterization, an approximation to the gradient can be computed by reusing information from the objective function. The details of the computation of the gradient are given in Algorithm~\ref{alg:MCobjgrad}. 

\subsection{Discussion on computational cost/accuracy}
A summary of the computations involved in the evaluation of the objective function and the gradient for a given $\bftheta$ is present in Algorithm~\ref{alg:MCobjgrad}.

\begin{algorithm}[!ht]
\begin{algorithmic}[1]
    \REQUIRE Matrix $\bfA$, vector $\bfd$, instance of hyperparameter $\bftheta$
    \STATE Compute matrices (or appropriate function handles) $\bfR(\bftheta)$, $\bfQ(\bftheta)$, $\bfPsi(\bftheta)$, and vector $\bfmu$ 
    \STATE Compute the derivatives (or appropriate function handles) $\{ \frac{\partial \bfPsi}{\partial \theta_i}\}_{i=1}^K$ and $\{\frac{\bfmu}{\partial \theta_i}\}_{i=1}^K$
    \STATE Compute $[\bfW,\bfD]=$Precond($\bfA,\bftheta$); matrices $\bfW,\bfD$ define the matrix $\bfG$  
    \STATE $\text{ld}_{\bfG} = -\frac12 \logdet(\bfR) + \logdet(\bfI -\bfD)$
    \STATE \COMMENT{Stage 1: Objective function computation}
    \STATE Initialize $\text{ld} \leftarrow 0$
    \FOR {$t=1,\dots,n_{\rm mc}$}
    \STATE Run Lanczos (using Algorithm~\ref{alg:lanczos}): $[\bfV_k,\bfT_k] = \text{Lanczos}(\bfPsi, \bfG, \bfw_t, k)$
    \STATE $\text{ld} \leftarrow \text{ld} + \frac{1}{n_{\rm mc}}\|\bfw_t\|_2^2 \bfe_1\t\log(\bfT_k)\bfe_1$
    \STATE Compute $\bfzeta_t = \|\bfw_t\|_2\bfG^{-\top}\bfV_k \bfT_k^{-1/2}\bfe_1 $
    \ENDFOR 
    \STATE Compute $\bfr = \bfA\bfmu -\bfd$ and solve $\bfPsi\bfz = \bfr $ using preconditioned CG with preconditioner $\bfG$
    \STATE $\widehat{\mc{F}}_{\rm prec} = -\log \pi_{\rm hyp}(\bftheta) + \frac{1}{2}(\text{ld} - 2\text{ld}_{\bfG}) + \frac12\bfz\t\bfr$
    \STATE \COMMENT{Stage 2: Gradient evaluation}
    \FOR{$i=1,\dots,K$}
    \STATE $ (\widehat{\nabla \mc{F}})_i = - \frac{1}{\pi_{\rm hyp}(\bftheta)} \frac{\partial \pi_{\rm hyp}(\bftheta)}{\partial \theta_{i}} + \frac{1}{2n_{\rm mc}} \sum_{j=1}^{n_{\rm mc}} \bfzeta_t\t  \frac{\partial \bfPsi}{\partial \theta_i} \bfzeta_t -   \frac12[\bfz\t \frac{\partial \bfPsi}{\partial \theta_i} \bfz - (\bfA\t\bfz)\t\frac{\partial \bfmu}{\partial \theta_i}]$
    \ENDFOR
    \RETURN Approximation to objective function $\widehat{\mc{F}}_{\rm prec}$ and approximation to gradient $\{(\widehat{\nabla \mc{F}})_i\}_{i=1}^K$
\end{algorithmic}
\caption{Monte Carlo estimator for the objective function and the gradient}
\label{alg:MCobjgrad}
\end{algorithm}

To determine the computational cost, we have to make certain assumptions. We assume that the cost of computing matvecs with $\bfA$ and its transpose is $T_{A}$ flops. Similarly, the cost of computing matvecs with $\bfQ$ and its derivatives is $T_Q$ flops. Since $\bfR$ is assumed to be diagonal, the cost of a matvec with it or its inverse is $\mc{O}(m)$ flops. With these assumptions, a matvec with $\bfPsi$ is $T_\Psi = 2T_A + T_Q + \mc{O}(m)$ flops. A discussion of the computational costs follows:

\begin{enumerate}
    \item \textbf{Precomputation}:
With these assumptions, the cost of the precomputation involved in constructing the preconditioner is $\mc{O}(nr^2) + T_Ar$ flops.
\item \textbf{Objective function}: We can compute the preconditioner $\bfG$ in $\mc{O}(mr^2 + r^3)$ flops. The basis $\bfV_k$ can be computed in $kT_\Psi + \mc{O}(mk^2)$ flops, when Lanczos is used with reorthogonalization. In total, $\text{ld}$ can be computed in $  n_{\rm mc}( kT_\Psi   +  \mc{O}(mk^2+k^3))$ flops. There is an additional cost of computing $\bfz$ and $\bfz\t\bfr$, which is $kT_\Psi + \mc{O}(mk)$ flops.

\item \textbf{Gradient computation}: Re-using information from the objective function approximation, we can evaluate the gradient using  $T_\Psi (n_{\rm mc}+1) K$ additional flops.

\end{enumerate}

\section{Numerical Experiments}
\label{sec:numerics}
In this section, we investigate the performance of the Monte Carlo estimator for hyperparameter estimation. We consider two examples from seismic inversion, including a static problem and a dynamic problem.
To report computational time, we use MATLAB R20234b on a MacBook Pro with Apple M1 chip, with 16GB memory, and use 8 workers. For the optimization, we use the \verb|fmincon| solver in MATLAB which uses the interior point method and nonnegativity constraints. 

Note that the total number of Lanczos iterations per function (and gradient) evaluations is reported. This is subject to different stopping criteria. 
The Lanczos iterations for the evaluation of the matrix logarithms in \eqref{eqn:lanczoslog} are stopped when the absolute value of the relative difference between successive iterations, i.e.,
$$ |\bfe_1\t \log(\bfT_k)\bfe_1 - \bfe_1\t \log(\bfT_{k-1})\bfe_1| \slash |\bfe_1\t \log(\bfT_k)\bfe_1|$$
falls below a threshold: in particular, this threshold is set to be $10^{-7}$. 
 We solve linear systems with $\bfPsi$ using the Lanczos method, where the iterations are stopped when the corresponding relative residual norm falls below a threshold: in particular, this threshold is set to be  $10^{-8}$. 

\subsection{Static seismic inversion}
We consider a model problem from seismic inversion, where the goal is to image the slowness of the subsurface using seismic waves. We construct an instance of the problem using the \verb|PRseismic| function in IR Tools~\cite{gazzola2019ir}. The number of unknown pixels is $n = 256^2 = 65, 536$. The number of measurements $m$ ranges from $1,440$ to $12,960$ and is obtained from $m = (32*j)(45*j)$, where $32*j$ is the number of source rays and $45*j$ is the number of receivers and $j\in \{1,2,3\}$. To simulate measurement noise, we add $2\%$ Gaussian white noise. We take $\bfR(\bftheta) = \sigma_m^2 \bfI$, $\bfmu(\bftheta) = \bfzero$, and the covariance matrix $\bfQ(\bftheta)$ based on the Mat\'ern covariance kernel, 
\[ {\rm matern}(\bfx,\bfy) = \sigma_n^2 \frac{2^{1-\nu}}{\Gamma(\nu)}\left( \sqrt{2\nu}\frac{\|\bfx-\bfy\|_2}{\ell} \right)^\nu K_\nu\left( \sqrt{2\nu}\frac{\|\bfx-\bfy\|_2}{\ell}\right),\]
where $K_\nu$ is the modified Bessel function of the second kind, $\sigma_n^2$ controls the variance of the process and $\ell$ is the length scale. We take $\nu \in \{\frac12,\frac32,\frac52\}$ and the hyperparameters vector to be $\bftheta=(\sigma_m^2, \sigma_n, \ell)$, so that $K=3$.

\paragraph{Accuracy of objective function}
We first investigate the accuracy of the estimator to the objective function with an increasing number of Monte Carlo vectors $n_{\rm mc}$. The number of measurements is taken to be $m = 1440$, the parameter $\nu = 1/2$, and the preconditioner rank is $r = 20^2$. The initial value of $\bftheta$ is taken to be $ \bftheta_0 = (10^{-3}, 0.8147, 0.9058)$ where the last two coordinates are generated at random. For this example, the true value of $\sigma_m^2$ is $1.0784\times 10^{-4}$. We report the relative error in the objective function $|\mc{F}-\widehat{\mc{F}}_{\rm prec}|/|\mc{F}|$, which we average over $10$ independent runs per sample size. In Figure~\ref{fig:mcaccuracy} (left panel), we plot the relative error versus the number of Monte Carlo vectors, $n_{\rm mc}$, which varies from $8$ to $200$ in increments of $16$. The solid line denotes the average over $10$ runs, and the error bars represent $1$ standard deviation. As can be seen, the mean of the error decreases slowly with increasing $n_{\rm mc}$.
\begin{figure}[!ht]
    \centering
    \includegraphics[scale=0.32]{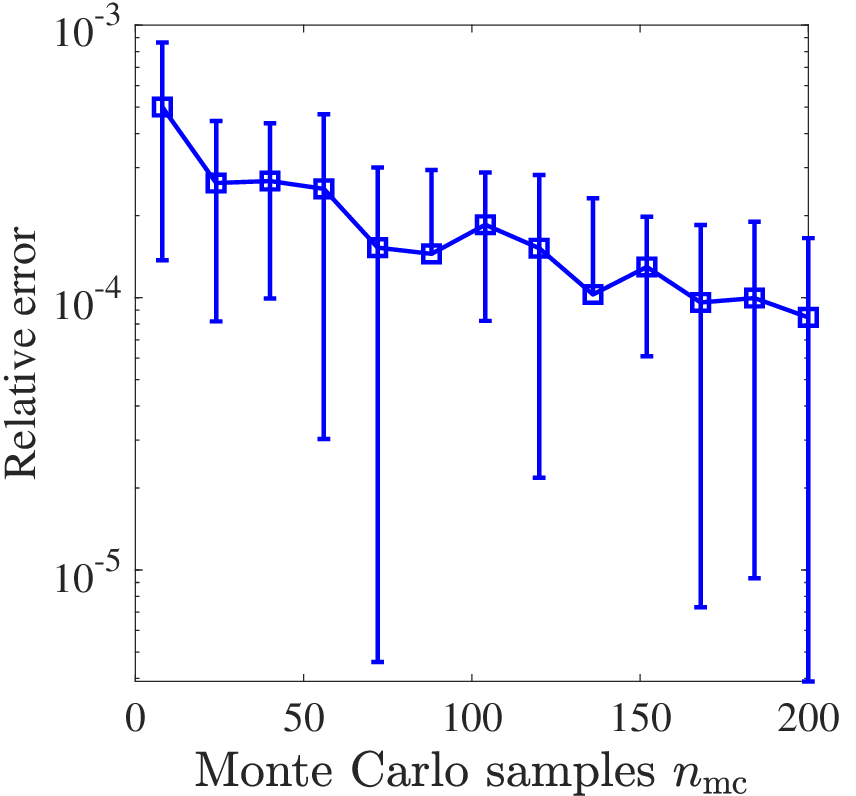}
    \includegraphics[scale=0.32]{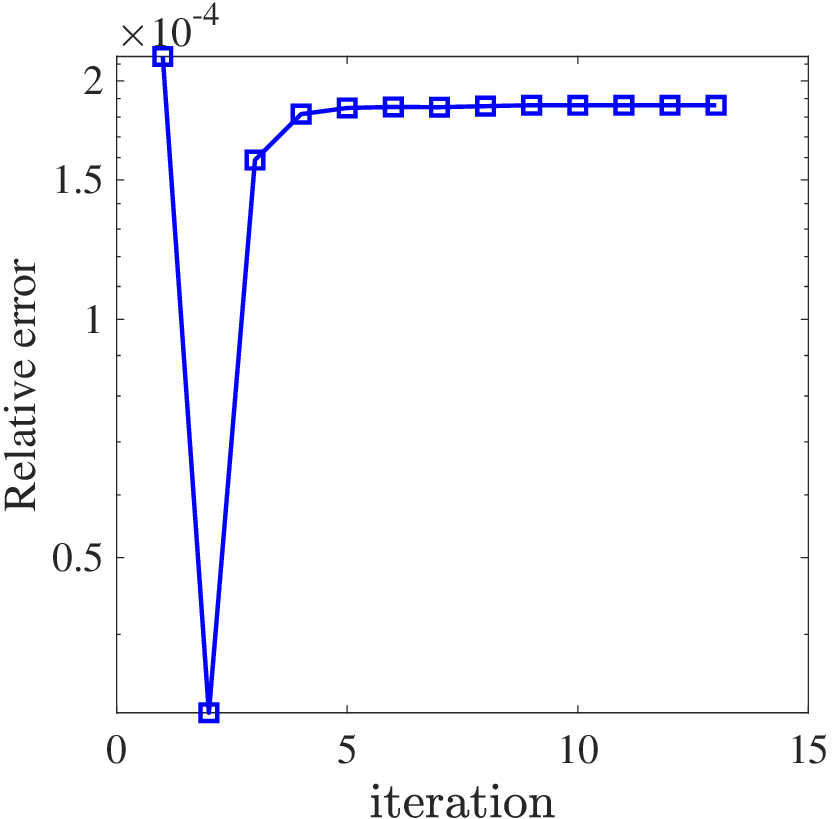}
    \caption{In the left panel are relative errors of the objective function provided in terms of the number of Monte Carlo samples.  The solid line corresponds to the average over 10 runs, and the bars represent $1$ standard deviation.  In the right panel, relative objective function errors at each iteration of the optimization scheme are provided.}
    \label{fig:mcaccuracy}
\end{figure}
 
 In the right panel of Figure \ref{fig:mcaccuracy}, we plot the history of the relative error across the optimization routine. The optimizer took 12 iterations and 81 function evaluations for convergence, and we observe that the error remains comparable throughout the iteration history.
 The image reconstructions with initial hyperparameters $\bftheta_0$ and optimized hyperparameters $\bftheta_{\rm prec} = \{ 9.870\times 10^{-5},
   1.5101,
  58.3143\}$ are given in Figure \ref{fig:reconstructions}. The ground truth image is provided in the left panel of Figure \ref{fig:reconstructions} for reference. Further numerical experiments below give insight into the performance of the preconditioner and the Monte Carlo estimators. 
   \begin{figure}[!ht]
    \centering
    \includegraphics[scale=0.4]{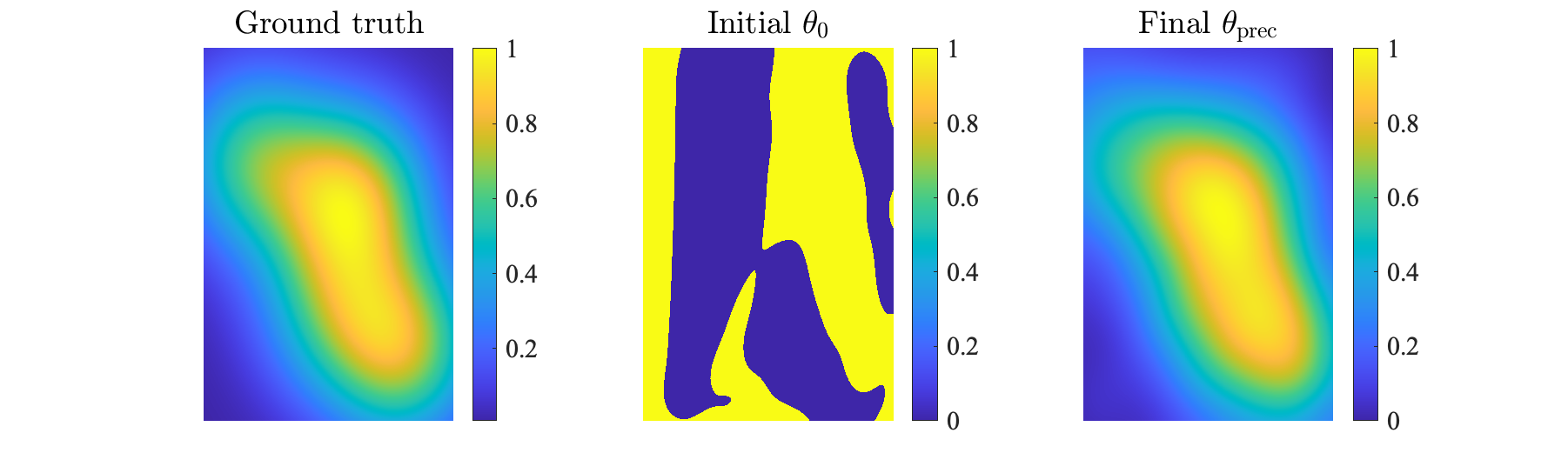}
    \caption{In the left panel is the ground truth image for the static seismic inverse example.  The middle and right panels contain image reconstructions corresponding to the initial hyperparameters $\bftheta_0$ and the optimized hyperparameters $\bftheta_{\rm prec}$ respectively.}
    \label{fig:reconstructions}
\end{figure}

\paragraph{Increasing preconditioner rank}
In this experiment, we take the number of Monte Carlo vectors $n_{\rm mc } = 24$. We choose the Mat\'ern parameter $\nu \in \{\frac12,\frac32,\frac52\}$ and study the effect of the preconditioner rank, which we vary from $r=5^2$ to $r=20^2$. The value of $\bftheta$ is taken to be $\bftheta_0$. We report in Table~\ref{tab:precondrank} the relative error in the objective function and the average number of Lanczos iterations required to compute the objective function. 
\begin{table}[!ht]
    \centering
\begin{tabular}{|c|c|c|c|c|}
\multicolumn{5}{c}{Accuracy of Monte Carlo estimator}\\ \hline
\diagbox[width=10em]{$\nu$}{$r$} & $5^2$ & $10^2$ & $15^2$ & $20^2$ \\ \hline
  $1/2$ & 9.8741e-04 & 1.5955e-03 & 1.6459e-04 & 2.1472e-04 \\ 
$3/2$ & 2.9869e-04 & 1.9522e-04 & 1.2817e-05 & 1.6093e-05 \\
$5/2$ & 1.2456e-04 & 1.2348e-05 & 1.1896e-07 & 1.9369e-07 \\
\hline 
  \multicolumn{5}{c}{Average number of Lanczos iterations}\\ \hline
    \diagbox[width=10em]{$\nu$}{$r$} & $5^2$ & $10^2$ & $15^2$ & $20^2$ \\ \hline
$1/2$ &  53.88 & 37.25 & 23.96 & 16.29 \\ 
$3/2$ & 22.79 & 10.50 & 6.04 & 4.42 \\
$5/2$ & 13.50 & 5.33 & 3.96 & 3.12 \\
     \hline
    \end{tabular}
    \caption{These results provide a comparison of the relative errors of the Monte Carlo estimator of the objective function, as well as the average number of Lanczos iterations, for different preconditioner ranks. 
    }
    \label{tab:precondrank}
\end{table}

We first comment on the accuracy. As the rank increases, the errors decrease (on average) for a fixed parameter $\nu$. For a fixed rank, the accuracy increases with increasing parameter $\nu$. Both observations can be explained by the fact that the preconditioner is becoming more effective with increasing rank and with increasing $\nu$. The effectiveness with higher rank is easy to understand, since a higher rank means a better approximation to the covariance matrix $\bfQ(\bftheta)$. To explain the effectiveness with higher $\nu$, observe that the eigenvalues of $\bfQ(\bftheta)$ decay more sharply with increasing $\nu$, since the kernel is smoother. Therefore, for the same rank, the preconditioner is more effective for a larger value of $\nu$. Next, we observe that the number of Lanczos iterations decrease with increasing rank and decrease with increasing $\nu$. Both of these observations reinforce the effectiveness of the preconditioner.

\paragraph{Different values of noise variance}
In this experiment, we take the number of Monte Carlo vectors $n_{\rm mc } = 24$ and fix the preconditioner rank to $r=20^2$. We choose the Mat\'ern parameter $\nu \in \{\frac12,\frac32,\frac52\}$ and study the effect of the preconditioner as the value of the noise variance estimate $\sigma_m^2$ ranges from $10^{-6}$ to $10^2$. The other two parameters $\sigma_n$ and $\ell$ are fixed as before.  The relative errors for the objective function as well as the average number of Lanczos iterations are reported in Table~\ref{tab:precondtheta1}. 
\begin{table}[!ht]
    \centering
\begin{tabular}{|c|c|c|c|c|c|}
\multicolumn{6}{c}{Accuracy of Monte Carlo estimator}\\ \hline
 \diagbox[width=10em]{$\nu$}{$\sigma_m^2$} & $10^2$ & $10^0$ & $10^{-2}$ & $10^{-4}$ & $10^{-6}$ \\ \hline
   $1/2$ &  9.4747e-08 & 1.5514e-03 & 2.2132e-04&  3.7249e-04 & 2.3660e-03 \\
$3/2$ &  3.3427e-10 & 8.7440e-06 & 2.8728e-06 & 7.2372e-05 & 2.3630e-05  \\
$5/2$ & 2.9519e-12 & 8.4275e-08 & 2.9167e-08 & 1.5431e-06  & 6.4102e-06 \\
   \hline 
  \multicolumn{6}{c}{Average number of Lanczos iterations}\\ \hline
    \diagbox[width=10em]{$\nu$}{$\sigma_m^2$}& $10^2$ & $10^0$ & $10^{-2}$ & $10^{-4}$ & $10^{-6}$\\ \hline
    $1/2$ & 3.08 & 4.12 & 8.17 & 45.17 & 247.54 \\
$3/2$ & 3.08 & 3.08&  3.62 & 6.54 & 38.62 \\
$5/2$ & 107.58 & 4.58 & 3.08 & 3.88 & 7.17 \\ \hline
    \end{tabular}
    \caption{These results provide a comparison of the accuracy of the Monte Carlo estimator for the objective function, as well as the average number of Lanczos iterations, for different values of $\sigma_m$ (the noise variance).
    }
    \label{tab:precondtheta1}
\end{table}

We observe that with decreasing estimates of $\sigma_m^2$, the accuracy decreases and the number of Lanczos iterations increase. This is because the preconditioner is less effective for smaller values of $\sigma_m$, since the preconditioner becomes closer to being singular. Furthermore, since $\sigma_m^2$ shifts all the eigenvalues of $\bfA\bfQ(\bftheta)\bfA\t$ away from zero, smaller values of $\sigma_m^2$ may lead to larger condition numbers and, therefore, require more Lanczos iterations. As before, we see that the accuracy increases and the number of Lanczos iterations decreases with increasing $\nu$, since the preconditioner is becoming more effective.

\paragraph{Timing} We now fix the number of Monte Carlo vectors $n_{\rm mc} = 24$ and the parameter $\nu = 3/2$ and we increase the number of measurements from $1,440$ to $12,960$. The rank of the preconditioner is taken to be $r=20^2=400$. For one function and gradient evaluation, at fixed $\bftheta_0,$ we provide in Table~\ref{tab:incmeas} the average number of Lanczos iterations required by the objective function computations, the wall-clock time (in seconds) taken by the Monte Carlo estimators and the time taken for the full evaluations (Full). 
\begin{table}[!ht]
    \centering
    \begin{tabular}{c|c|c|c}
        \# Measurements & Average number of iterations & MC [sec] & Full [sec]   \\ \hline
        $1,440$ & $11.25$ & $3.85$ & $34.82$  \\
        $5,760$ & $16.38$ & $4.05 $ & $157.05$\\
        $12,960$ & $22.25$ & $7.08$ & $-$
    \end{tabular}
    \caption{Comparison of number of Lanczos iterations (averaged over 24 Monte Carlo vectors) and timings in seconds for one function and one gradient evaluation, for different numbers of measurements. }
    \label{tab:incmeas}
\end{table}

We observe that the number of Lanczos iterations shows a mild growth with the number of measurements. The increase in measurements brings in additional information that the iterative solver has to resolve. Correspondingly, there is a mild growth in the run time for the Monte Carlo estimator. Nevertheless, in all scenarios, obtaining the Monte Carlo estimate is significantly faster than evaluating the full function and gradient, if it is even possible.

\subsection{Dynamic seismic inversion}\label{sec:numerics_dynamic}
We now extend the previous example to include a temporal component in the solution. In particular, the true solution is a spatial-temporal object consisting of two rotating Gaussians. We use the \verb|PRseismic| function in IR Tools~\cite{gazzola2019ir} to model the measurement process at each time point, where the number and location of the sources and the receivers do not change over time. We consider $128^2$ spatial locations at $50$ time points, so the number of unknown pixels is $n = 819, 200$.  We have $20$ receivers and sources resulting in $m = 20, 000$ measurements ($400$ per time point). The exact solution and the measurements at subsampled time points are provided in Figure~\ref{fig:dynamic_setting}.

\begin{figure}[!ht]
    \center
    \includegraphics[scale=0.8]{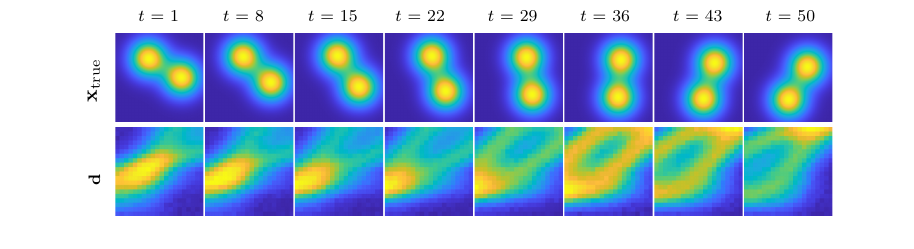}
    \caption{Top row: exact solutions at timepoints $t$ consisting of two rotating Gaussians. Bottom row: noisy measurements. The proportions of the images are accurate but, to aid visualization, the relative size between images and measurements is not.}
    \label{fig:dynamic_setting}
\end{figure}

Note that adding a temporal component to the solution  increases the size of the problem considerably, making exact function and gradient evaluations prohibitively costly. We assume that the forward operation at different time points is the same, resulting in a Kronecker product structure for the system matrix.  This allows for efficient matvecs with $\bfA$. In this example, the system matrix $\bfA$ is a Kronecker product of a matrix corresponding to a static seismic problem, $\bfA_s$, and the identity matrix: $\bfA = \bfI \otimes \bfA_s$. As in the previous example, we add $2\%$ Gaussian white noise so that $\bfR(\bftheta) = \sigma_m^2 \bfI$. The true value of $\sigma_m^2$ is $0.5915$.  Moreover, we assume that $\bfmu(\bftheta) = \bfzero$ and the covariance matrix also has Kronecker structure $ \bfQ(\bftheta) =  \bfQ_t(\bftheta) \otimes \bfQ_s(\bftheta) $ where $\bfQ_s(\bftheta)$ arises from a Mat\'ern covariance kernel with variance $\sigma_n^2$, $\nu=3/2$, and length scale $\ell_s$; similarly, $\bfQ_t(\bftheta)$ arises from a Mat\'ern covariance kernel with variance $1$, $\nu=5/2$, and length scale $\ell_t$. 
Note that we only consider a single parameter $\sigma_n^2$ controlling the  variance of the prior.  Therefore, the hyperparameters vector for this example is $\bftheta=(\sigma_m^2, \sigma_n, \ell_t, \ell_s)$ and $K=4$, where we assume the hyperprior defined in \eqref{eq:hyperprior}. The preconditioner is constructed using separate approximations of the temporal and spatial covariance matrices, with the ranks being $5$ and $10$ respectively, so that 
$$\bfP = \bfP_t \otimes \bfP_s = \bfU_t \bfM_t(\bftheta) \bfU_t\t  \otimes \bfU_s \bfM_s(\bftheta) \bfU_s\t =( \bfU_t \otimes \bfU_s) (\bfM_t(\bftheta) \otimes \bfM_s(\bftheta))(\bfU_t  \otimes \bfU_s)\t.$$ The approximations of the time and space covariance matrices are constructed independently as explained in Section \ref{ssec:precond}. The number of Monte Carlo vectors is taken to be $n_{\rm mc} = 18$.

The reconstructions at a subset of times and using different hyperparameters can be observed in Figure~\ref{fig:reconstructionsdyn}. In the top row, we provide the reconstructions using the initial hyperparameters, $\bftheta_{0}=(0.5915  \,1\,1\,1)$.  In the second row are the reconstructions using the optimal hyperparameters obtained using the proposed SAA optimization with function and gradient approximations coming from the Monte Carlo estimator without preconditioning, and finally, the last row contains the reconstructions with hyperparameters obtained with preconditioning. In the bottom of each figure, we provide the relative error norm of the reconstruction at that time point. One can clearly observe the improvement on the reconstructions when using the optimal hyperparameters computed using preconditioning. This example illustrates the importance of accurate Bayesian modeling to faithfully recover spatial features in noisy, underdetermined problems.  This example also illustrates the challenges with hyperparameter estimation and the need to have accurate enough function and gradient evaluations of the marginal posterior.

\begin{figure}[!ht]
    \centering
    \includegraphics[scale=0.8]{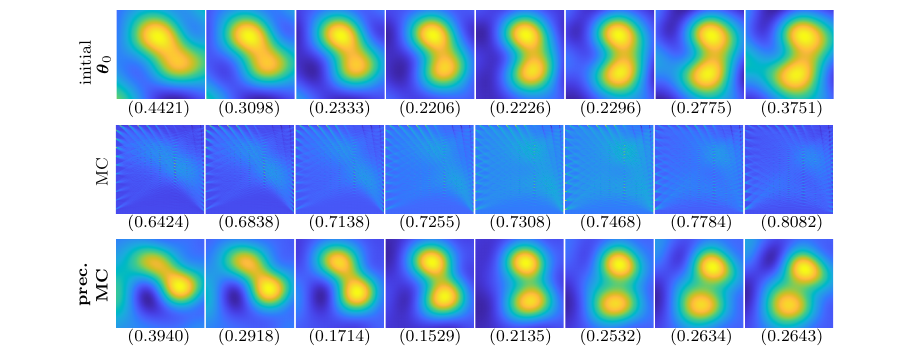}
    \caption{Reconstructions using the seismic inversion model corresponding to the initial hyperparameters $\bftheta_{0}=(0.59  \,1\,1\,1)$ (top row), the optimal parameters computed using approximate function and gradient evaluations based on the new SAA methods without preconditioning $\bftheta_{\bf mc}~=~(0.15  \,647.51\, 0.21\,0.00)$ (middle row), and the optimal parameters computed with preconditioning $\bftheta_{\bf prec}=(0.57  \,  0.19  \,  0.47  \,  0.32)$ (bottom row). }
    \label{fig:reconstructionsdyn}
\end{figure}

For additional insight into the Monte Carlo approaches, we provide in Table~\ref{table:dynamic} the number of iterations required in the optimization procedure, the total number of function evaluations, the average number of Lanczos iterations, and the number of times we reached the maximum number of iterations (i.e., when tolerances were not met). From the first two columns of Table~\ref{table:dynamic}, we observe that without preconditioning, fewer iterations and function evaluations were required.  
Note that the size of these problems made it prohibitive to compare the accuracy of the function (and gradient) evaluations with respect to the full evaluations. However, from the last two columns of Table~\ref{table:dynamic}, we can still say something about the convergence of the evaluations in terms of the preconditioning within the Lanczos iterations. First, note that the number of Lanczos iterations corresponds to the computation of the log determinant, and that the maximum number of iterations per evaluation was set to 350. 
We observe that the average number of Lanczos iterations is much lower when using preconditioning, and it is likely that the difference would have been greater if the maximum number of iterations was set higher. Moreover, the number of times that the Lanczos iterations reached the maximum amount allowed is much higher for the unpreconditioned case (36 out of 44 evaluations) compared to the preconditioned case (4 out of 87 evaluations, all occurring at the start of the optimization). In these cases, where the stopping criterion was not met, it is more likely that the function approximations are of much poorer quality. This is a possible explanation for the poor reconstructions in the unpreconditioned case.
\begin{table}[!ht]
    \begin{tabular}{l|l|l|l|l}
    \centering
    & iterations & evaluations & Lanczos iterations & reached max  \\
    \hline 
with preconditioning    & 18 & 87  & 122                         &   4   \\
without preconditioning & 8 & 44  & 303                      &     36        \end{tabular} 
\caption{In the first and second columns, we provide the number of iterations of the optimization procedure, as well as the total number of function evaluations, respectively.  In the third column, we provide the average number of Lanczos iterations involved in the computation of the log determinant; averaged both over the number of Monte Carlo samples and over all function evaluations. In the last column, we provide the number of instances where the function evaluations hit the maximum number of Lanczos iterations, which was set to 350 iterations.}
\label{table:dynamic}
\end{table}

\section{Conclusions/Discussion}
\label{sec:conclusions}
In this work, we develop an efficient and robust approach for hyperparameter estimation in hierarchical Bayesian inverse problems.  By exploiting a Monte Carlo approach, combined with an efficient preconditioned Lanczos method, the proposed SAA approach enables objective function and gradient estimations of the negative log marginal posterior.  We show how these approximations can be used within constrained optimization schemes for hyperparameter estimation. Future work includes investigating other types of trace estimators in this framework and combining alternative preconditioners, e.g., those that can exploit problem structure. Extensions to nonlinear forward models and non-Gaussian priors are also of interest.
\appendix

\section{Parametric kernel low-rank approximation for preconditioning}\label{app:paramlowrank}
We explain the construction of the parametric kernel low-rank approximation technique, the details of which are given in~\cite{fong2009black,khan2024parametric}. For completeness, we offer a brief description but refer the reader to the above resources for a more complete description. 

Our goal is to obtain an approximation of the form $\bfQ(\bftheta) \approx \bfU\bfM(\bftheta) \bfU\t$ given a set of points $\{\bfx_j\}_{j=1}^n$ and the kernel function $\kappa(\bfx,\bfy;\bftheta)$. Let the points $\{\bfx_j\}_{j=1}^n$ be enclosed in a box $\mc{B} = [\alpha_1, \beta_1] \times \dots \times [\alpha_d,\beta_d] \subset \bbR^d$. Similarly, let the parameters $\bftheta \in \mc{B}_\theta = [\alpha_1^p, \beta_1^p] \times \dots \times [\alpha_K^p,\beta_K^p] \subset \bbR_+^K$ be enclosed inside a box. For a nonnegative integer $j$, the $j$th Chebyshev polynomial of the first kind is $T_j(x) = \cos(j\arccos(x))$ for $-1 \le x \le 1$.   To transform to and from the interval $[\alpha,\beta]$, we use the invertible transform $\varphi_{[\alpha,\beta]}(x) = \frac{\beta-\alpha}{2}x  + \frac{\alpha+\beta}{2}$. We also define the interpolating polynomial 
\[\phi_{[\alpha,\beta]}(x,y) = \frac1p + \frac2p \sum_{j=1}^p T_j(\varphi_{[\alpha,\beta]}^{-1}(x))T_j(\varphi_{[\alpha,\beta]}^{-1}(y)), \]
and the Chebyshev interpolation nodes (for $1\le i \le p$)
\[ \eta_i^{(j)} = \left\{ \begin{array}{ll}
    \varphi_{[\alpha_j,\beta_j]}(\zeta_i)  &  1\le j \le d  \\
    \varphi_{[\alpha_{j-d},\beta_{j-d}]}(\zeta_i)  &  d+1\le j \le d +K \\
    \varphi_{[\alpha_{j-(d+d_\theta)},\beta_{j-(d+d_\theta)}]}(\zeta_i)  &  d+d_\theta+1\le j \le 2d+K. 
\end{array} \right. \]
The roots of $T_p$ define the Chebyshev nodes $\{\zeta_j\}_{j=1}^p$.
The factor matrices $\bfU$ can be expressed as $\bfU = \bfU_d \ltimes \dots \ltimes \bfU_1$, where
\[ [\bfU_j]_{it} = \phi_{[\alpha_j,\beta_j]}(\eta_t^{(j)},[\bfx_i]_j) \qquad 1 \le i \le n, 1 \le t \le p,  1\le j \le d, \]
and $\ltimes$ denotes the face splitting product. 

To define the matrix $\bfM(\bftheta)$, we define the vector of interpolating polynomials
\[ \bfq_i(\theta_i) = \bmat{\phi_{[\alpha_i^p,\beta_i^p]}(\eta_1^{(d+i)},\theta_i) & \dots & \phi_{[\alpha_i^p,\beta_i^p]}(\eta_n^{(d+i)},\theta_i)} \in \R^{1\times p}, \qquad 1 \le i \le K.\]
We also define the tensor $\ten{M}$, of size $p^D$ where $D = 2d + K$, and with entries 
\[ m_{i_1,\dots,i_D} = \kappa(\bfx, \bfy; \bftheta) \qquad 1 \le i_j \le p, 1 \le j \le D.\] 
Furthermore, $\bfx,\bftheta$ and $\bfy$ are defined via the relation $(\bfx, \bftheta, \bfy) = (\eta_{i_1}^{(1)}, \dots , \eta_{i_D}^{(D)})$. Next, we define the parametric tensor  $\ten{M}_F(\bftheta)$ via the tensor products 
\[ \ten{M}_F(\bftheta) = \ten{M} \times_{i=d+1}^{d+K} \bfq_{i-d}(\theta_{i-d}).\]
Finally, the matrix $\bfM(\bftheta)$ is obtained from a specific mode unfolding of the tensor as $\bfM_F^{(d)}(\bftheta)$. Details are given in~\cite{khan2024parametric}.

\section*{Funding}{This work was partially supported by the National Science Foundation program under grants DMS-2411197, DMS-2208294, DMS-2341843, DMS-2026830, DMS-2411198, and DMS-2026835. Any opinions, findings, conclusions or recommendations expressed in this material are those of the author(s) and do not necessarily reflect the views of the National Science Foundation. For the purpose of Open Access, the author has applied a CC BY public copyright licence to any Author Accepted Manuscript (AAM) version arising from this submission.}

\bibliography{refs}
\bibliographystyle{abbrv}
\end{document}